%% file: matema-english.tex
\documentclass[a4paper,12pt,twoside]{report} 
\usepackage{a4wide}
\usepackage{
	latexsym,
	amssymb,
	amsmath,
	fixltx2e,
	t1enc,
	amsthm,
	amscd,
	delarray
}
\title{Effective Banach spaces}
\author{Bj{\o}rn Kjos-Hanssen}
\begin{document}
	\maketitle
	\newtheorem{Kor}{Korollar}
	\newtheorem{Theorem}[Kor]{Theorem}
	\newtheorem{Lemma}[Kor]{Lemma}
	\newtheorem{Def}[Kor]{Definisjon} 
	\newtheorem{Bem}[Kor]{Bemerkning}
	\newtheorem{Prop}[Kor]{Proposisjon}
	\newtheorem{Oppgave}[Kor]{Oppgave}
	\newtheorem{Eksempel}[Kor]{Eksempel}
	\newcommand{\N}{\mathbb{N}}
	\newcommand{\R}{\mathbb{R}}
	\newcommand{\Q}{\mathbb{Q}}
	\newcommand{\C}{\mathbb{C}}
	\newcommand{\F}{\mathbb{F}}
	\newcommand{\og}{\wedge}
	\newcommand{\eller}{\vee}
	\newcommand{\ikke}{\neg}
	\newcommand{\snitt}{\cap}
	\newcommand{\union}{\cup}
	\newcommand{\hviss}{\leftrightarrow}
	\newcommand{\Hviss}{\Leftrightarrow}
	\newcommand{\saa}{\rightarrow}
	\newcommand{\Saa}{\Rightarrow}
	\newcommand{\delm}{\subseteq}
	\newcommand{\ikketom}{\neq \emptyset}
	\newcommand{\la}{\langle}
	\newcommand{\ra}{\rangle}
	\newcommand{\mb}{\mathbb}
	\newcommand{\mc}{\mathcal}
	\newcommand{\mf}{\mathfrak}
	\newcommand{\B}{\operatorname{B}}
	\newcommand{\op}{\operatorname}
	\newcommand{\Lin}{Line{\ae}rkombinasjon-aksiomet }
	\newcommand{\Norm}{Norm-aksiomet }
	\newcommand{\Lim}{Limes-aksiomet }

	\newcommand{\Gregor}{Grzegorczyk}
	\newcommand{\Rk}{\mb{R}_{\op{k}}}
	\newcommand{\Ck}{\mb{C}_{\op{k}}}
	\newcommand{\Xk}{\mf{X}_{\op{k}}}
	\newcommand{\Yk}{\mf{Y}_{\op{k}}}
	\newcommand{\Ak}{A_{\op{k}}}
	\newcommand{\Bk}{B_{\op{k}}}
	\newcommand{\Mk}{M_{\op{k}}}
	\newcommand{\BXk}{\B(\X)_{\op{k}}}
	\newcommand{\XNk}{\mf{X}^\N_{\op{k}}}
	\newcommand{\Bevis}{\paragraph{Bevis:}}
	\newcommand{\QED}{\ \ \(\Box\) \\}
	\newcommand{\X}{\mf{X}}
	\newcommand{\Y}{\mf{Y}}
	\renewcommand{\H}{\mc{H}}
	\newenvironment{Cases}{\begin{array}
	                       \{
	                      {cc}
	                       .}
	                      {\end{array}}

\tableofcontents
\begin{abstract}
	We show how Pour-El and Richards' computability structures on Banach spaces \(\X\) can be characterized in terms of 
	effective operators on the set of computable elements \(\Xk\).

	We give a computable, partial numbering of the set of effective operators \(\BXk\).
	Effective operators on \(\Xk\) have unique closed and continuous extensions to \(\X\).

	The non-separable operator space \(\B(\X)\) has an ineffective norm, unlike 
	the subspace \(\B_0(\X)\) of compact operators when \(\X\) has the approximation property.

	Totally defined numberings of \(\Xk\) do not give modulus of convergence. 
	A total numbering of \(\Xk\) is defined, which leads to a concept of quasi-effective operator. 
	\(x \mapsto \sqrt{2}x\) is not quasi-effective, so we conclude that partial numberings are preferable.
\end{abstract}


\newpage
\section*{Foreword}
	This Master's thesis in Mathematics at the Department of Mathematics, University of Oslo, 
	was written during the period August 24, 1996 to April 25, 1997 
	(cut a bit short since I started a PhD program that Berkeley). 
	The initial idea was to study problems arising from Pour-El and Richards' {\em Computability in Analysis and Physics} and 
	put their approach to computable functional analysis into a computability-theoretic context. 
	I thank Prof. Dag Normann for being interested in the topic and providing able weekly advice.

	Blindern, April 25 1997, (revised and translated Honolulu, 2012)
	Bj{\o}rn Kjos-Hanssen.

\chapter{Introduction}
	\input{introduction.tex}

\chapter{The scientific literature in the area}
	\input{literature.tex}

\chapter{Numberings}
	\input{numberings.tex}

\chapter{Effective operators}
	\input{effective-operators.tex}

\chapter{Operatorrom}
	\input{operator-spaces.tex}

\bibliography{h}
\bibliographystyle{alpha}
\end{document}

%% file: introduction.tex
\section{Notation}
	We consider vector spaces with scalar field \(\F\), where \(\F=\C\) or \(\F=\R\).
	Ordered pair is denoted \(\la \cdot, \cdot \ra\), whereas inner product is denoted \(( \cdot | \cdot )\).
	Some notation is imported from Pedersen \cite{Ped} and Odifreddi \cite{Odi}. The notation \(\Xk\) for the set of computable elements of \(\X\) is inspired by domain theory \cite{Dom}\cite{Blanck}. 
	To indicate what variable a function depends on in logic one writes e.g.\ \(\lambda x.f(x)\) and in mathematical analysis \(x \mapsto f(x)\). 
	The latter is preferable her to avoid associations with the eigenvalue equation \(Tx=\lambda x\). 
	In logic the set of natural numbers is often denoted \(\omega\), whereas here and in analysis one uses \(\N\).
	The expression \(X \equiv Y\) denotes definition, and makes sense when the meaning of exactly one of \(X\) and \(Y\) is already known or defined. 
	We also use the definition \(n \equiv_\alpha m \Hviss \alpha n = \alpha m\). 
	When function symbols are written in juxtaposition, as in \(fg\) or \(f n\), we mean, respectively, composition \(f \circ g\) and application \(f(n)\), whereas multiplication is denoted by \(f \cdot g\). 
	For a sequence \(\{ f(n) | n \in \N \}\) we may write \(( f(n) )\), \(\{ f(n) \}\) or simply \(f(\cdot)\).
	If a result is mentioned here that falls within pure computability theory or pure mathematical analysis, it is not new; 
	if it falls within their intersection it may be assumed to be new unless otherwise indicated.
	If \(f:A \saa B\) is a function between two sets, we let \(\op{ran} f = \{ y \in B | (\exists x \in A)(f(x)=y) \}\). The characteristic function of a set \(A\) is denoted \(c_A\) or \(\chi_A\).
	Let \(\ell_2\) be the Hilbert space of complex sequences \((\xi_n)\) such that \(\sum |\xi_n|^2 < \infty\).
	Let \(\delta\) denote the Kronecker delta, i.e., \(\delta_{jk}=1\) if \(j=k\) and \(0\) otherwise.
	The linear span of a set of vectors \(S\) is the set of all finite linear combinations from \(S\). 
	The rational span of \(S\) is the set of all linear combinations from \(S\) using scalars in \(\Q\), or \(\Q+i\Q\) in the complex case. 
	We remind the reader that \(s X = \{ s x | x \in X \}\) is a common form of notation in analysis. 
	Let \(\mc{R} \delm \mc{PR} \delm \mc{P}\) be the set of total recursive, partial recursive, and partial, functions from \(\N\) to \(\N\), respectively.

	\paragraph{Pour-El and Richards' fourth problem}

	\begin{Def}[\Gregor \ and Lacombe; the earlier Banach-Mazur version did not demand effective uniform continuity (in the linear case this is automatic)]
		\(f\) is computable if it send computable sequences to computable sequences and is effectively uniformly continuous, i.e., 
		for some computable function \(d:\N \saa \N\) we have \(|x-y| \leq 1/d(N) \S{\aa} |f(x)-f(y)| \leq 2^{-N}\).
	\end{Def}

	In the appendix ``Open Problems'' \cite{Pour} it is asked about a connection between higher recursion theory (HRT) and computable functional analysis.

	HRT deals with functionals of functions from \(\N\) to \(\N\), functionals of such functionals, etc. 
	A functional approach to computable analysis was given in \cite{Gre1}, where the real numbers are identified with the set \(R\) of functions \(\phi :\N \rightarrow \N\).

	Given a surjective map \(\nu:\N \saa \Xk\) of the computable elements \(\Xk\) of a Banach space \(\X\) (a {\em numbering} of \(\Xk\)) an operator \(F\) is called effective if
	there is a partial recursive \(f:\N \saa \N\) such that \((\forall e \in \op{dom} \nu)(F \nu e = \nu f e)\). 
	The partial recursive functions are defined inductively.
	One can also define the effective operators directly inductively, by stipulating that certain operators \(F, F', ..\) are to be effective
	and close this set under certain schemas that we say preserve effectivity.
	The latter occurs in higher recursion theory, but such schemas do not always preserve linearity.

\section{The category of Banach spaces}

	The class of all Banach spaces is a category \textbf{Ban}. For simplicity we mostly consider bounded linear operators from \(\X\) to \(\X\) instead of from \(\X\) to \(\Y\). 
	\(\op{Hom}(\X) = \B(\X)\) is the set of continuous linear operators on \(\X\).

	The product space is given the maximum norm \(\| (x,y) \| = \op{max} \{ \| x \|, \| y \| \}\).

%% file: literature.tex
\section{Computability structures on Banach spaces}

	In \cite{Pour} the concept of a computability structure on a Banach space is axiomatized, and the separable case is treated under the name ``effectively separable Banach space''. 
	We will assume separability and use the term ``effective Banach space''. This is in accord with the fact that effective metric spaces in the literature are assumed to be separable.
	We can still discuss non-separable Banach spaces with computability structures, as this notion is defined in \cite{Pour}.

	\paragraph{The axioms for a computability structure on a Banach space}

	The concept that is being axiomatized is ``computable sequence''. An {\em element} \(x \in \X\) is called computable if the constant sequence \((x,x,...)\) is computable.
	The computability structure may be viewed as the set of computable sequences, and an {\em effective Banach space} can then be defined as 
	a pair consisting of a Banach space \(\X\) and a computability structure on \(\X\). 
	If as in domain theory \cite{Dom} vi use the index \(\op{k}\) to denote the subset consisting of all computable elements, we may call a computability structure on \(\X\) for \((\X^\N)_{\op{k}}\).

	A Banach space is a complete normed vector space, and the axioms for computability structures appropriately concern limits, norms, and linear combinations.

	\paragraph{\Lin}

		Let \(\{ x_n \}\) and \(\{ y_n \}\) be computable sequences in \(\X\), let \(\{ \alpha_{nk} \}\) and \(\{ \beta_{nk} \}\) be computable double sequences of complex numbers, and 
		let \(d:\N \saa \N\) be a recursive function. Then the sequence
		\[ a_n = \sum_{k=0}^{d(n)}(\alpha_{nk} x_k + \beta_{nk} y_k) \]
		is computable in \(\X\).

	\paragraph{\Lim}
		Let \(\{ x_{nk} \}\) be a computable double sequence in \(\X\) such that \(\{ x_{nk} \}\) converges to \(\{ x_n \}\) as \(k \saa \infty\), effectively in \(k\) and \(n\). 
		Then \(\{ x_n \}\) is a computable sequence in \(\X\).

	\paragraph{\Norm}
		If \(\{ x_n \}\) is a computable sequence in \(\X\), then the sequence of norms \(\{ \|x_n\| \}\) is a computable sequence of real numbers.

	\begin{Def}
		Effective limits: If \(x_n\) converges effectively, let \(\lim^*_n x_n = \lim_n x_n\).

		The set of computable elements in a Banach space \(\X\) is denoted by \(\Xk\).

		The set of computable sequences from \(\X\) is denoted by \(\Xk^\N\).

		We say that a linear operator \(T\) {\em preserves computability} if it sends computable sequences to computable sequences: \((\forall(x_n) \in \Xk^\N)((Tx_n) \in \Xk^\N)\).

		A Banach space \(\X\) is called effective if there exists a computable sequence \(\{ e_n \}\) whose linear span is dense. 
		Such a sequence \(\{ e_n \}\) is called an effectively generating set (e.g.s.).
	\end{Def}

	\begin{Lemma}[Effective Density \cite{Pour}]
	If \(\{e_n\}\) is an e.g.s. in \(\X\), then a sequence \(x_n\) is computable iff we have \(x_n = \lim_k^* \sum_{j=0}^{d(n,k)} \alpha_{nkj} e_j\) 
	where \(d\) is a recursive function and \(\{ \alpha_{nkj}\}\) is a computable triple sequence. In particular an element \(x\) is computable if we can express it as 
	\(x = \lim_k^* \sum_{j=0}^{d(k)} \alpha_{kj} e_j\). 
	Thus a computability structure is uniquely determined by an e.g.s.
	\end{Lemma}

	\begin{Lemma}[The composition property \cite{Pour}]
	If \(\{ x_n \}\) is a computable sequence in \(\X\) and \(f:\N \saa \N\) is a recursive function, then \(\{ x_{f(n)}\}\) is a computable sequence in \(\X\).
	\end{Lemma}

	This becomes a trivial remark if we consider computable sequences to be sequences of the form \(\{ a_{f(n)} \}_{n \in \N}\) where \(f\) is recursive. 
	Then the composition property follows from the fact that the set of recursive functions is closed under composition.

	In \cite{Pour} one defines ``effectively determined operator'' as follows:

	A closed operator \(T\) on a Hilbert-space \(\H\) is called effectively determined if there exists a computable sequence \(\{ e_n \}\) in \(\H\) such that \(\{ (e_n, Te_n) \}\) is a e.g.s. for the graph of \(T\).

	Here we modify this definition.
	We only consider bounded operators (we shall se that all effective operators are bounded).
	Moreover we need the notion of an effectively determined operator on an arbitrary Banach space.

	\begin{Def}[Effectively determined operator]
	A bounded operator \(T\) on a Banach space \(\X\) is called effectively determined if there exists an e.g.s. \(\{ e_n \}\) in \(\X\) such that \(\{ Te_n \}\) is a computable sequence. 
	In particular the domain of \(T\) is dense.
	\end{Def}

	Whereas \Lin and \Lim constrain the size of a computability structure from below, \Norm constrain it from above,
	and connects this with a classical concept: computable sequences of real numbers. 
	Therefore, \Norm is the axiom that must be verified when trying to determine if a set of objects stipulated to be computable generates a computability structure.

	We require computability structures to be non-empty. By \Lin, it follows that the constant sequence \((0,0,..)\) is computable in every computability structure.

	Pour-El \& Richards prove a powerful theorem with applications that allow for answering many questions on the effectivity of physically relevant operators 
	-- but it also leads to new questions regarding the relationship between effectivity of an operator and the physical quantity represented by the operator.

	\begin{Theorem}[First Main Theorem, \cite{Pour}]
	Let \(\X\) and \(\Y\) be Banach space med computability structures. Let \(\{e_n\}\) be en e.g.s. in \(\X\). Let \(T:\X \saa \Y\) be a closed linear operator defined on \(\{ e_n \}\) 
	and such that \(\{ T e_n \}\) is a computable sequence in \(\Y\). If \(T\) are bounded sender \(T\) computable sequences on computable sequences. 
	If \(T\) is unbounded, there is an element \(x \in \Xk\) such that \(Tx \notin \Yk\).
	\end{Theorem}

	All the assumptions are satisfied for most physically relevant operators, so the unbounded ones among these are ineffective in a strong sense: 
	If \(T:\X \saa \X\) is unbounded then \(\Xk\) is not an invariant subspace for \(T\), 
	i.e. \(T[\Xk] \not\delm \Xk\).

	For general phenomena in computable analysis and in applications of the axiomatization to classical separable Banach space see \cite{Pour}.

\section{Effective metric spaces}

	Blanck \cite{Blanck} studied effective metric spaces using {\em Scott-Ershov-domains} \cite{Dom}.
	In recursion theory we often wish to reduce the uncountable to the countable in some sense. 
	The countable set of computable elements is of special interest, as most elements that are important in practice are computable. 
	For example one can define an algebraically closed field of computable complex numbers that contains all the numbers that one will encounter in practice outside of recursion theory. 
	When we define computable functions, we require that they treat computable input in a computable way.

	\begin{Def}
		Let \(\rho\) be en fixed standard enumeration of the computable real numbers \(\Rk\).
		A computable metric space is a pair \(((A,\alpha_0),d)\) where \(\alpha_0:\N \saa A\) is surjective, and the relation \(\equiv_{\alpha_0}\) given by
		\(m \equiv_{\alpha_0} n \Hviss \alpha_0 m = \alpha_0 n\) is recursive, and \(d:A^2 \saa \Rk\) is a \((\alpha_0,\rho)\)-computable metric, i.e. 
		there is a recursive \(\hat{d}\) such that \(\rho \circ \hat{d} = d \circ \alpha_0^2\).
	\end{Def}

	Let \(A^*\) be the closure of \(A\), and let \(\Ak\) be the effective closure of \(A\). If \(A \delm B \delm A^*\) and \(A\) is a computable metric rom, is called \(B\) an effective metric space.

	An element \(x \in A^*\) is called computable if there exists a \(\alpha_0\)-computable sequence \((a_n)\) from \(A\) such that \(\forall n \quad d(a_n,x)<2^{-n}\). 
	An index for \((a_n)\) is then called an \(\alpha\)-index for \(x\). We can then define \((x)_i = a_i\), the \(i\)th approximation of \(x\).
	\(\alpha_0\) is thus an enumeration of the countable dense subset \(A\), whereas \(\alpha\) is an enumeration of the effective completion \(\Ak\).

	We note quote one of Blanck's theorems, which uses common features between the numbering of \(\Ak\) and Blanck's construction of a domain representation of metric spaces 
	to make a decisive connection between standard resultats in domain theory, such as the Kreisel-Lacombe-Shoenfield theorem, and metric spaces.

	\begin{Theorem}
	Let \(A^*\) be an effective metric space, and let \(D\) be the constructed domain representation. There exists a recursive homeomorphism between the numbered set
	\(\Ak\) and a certain counterpart in \(D\).
	\end{Theorem}
	The proof uses the fact that a computable element is a limit of a computable geometrically convergent sequence, given with an index for that sequence.

	Blanck studies several substructures of \((\Rk,\rho)\), and questions on the effectivity in spaces that are countable unions of compact sets, i.e., locally compact and \(\sigma\)-compact spaces. 
	Many commonly studied spaces in the theory of linear operators such as (\(L^p\) and \(C([0,1]\) do not have this property.

	\begin{Def}
	A function \(f:M \saa M\) on an effective metric space \(M\) is called effectively continuous if there exists \(g \in \mc{R}\) such that for each basic open set \(B(\alpha e_0, \rho k_0)\), we have
	\[
		f^{-1}[B(\alpha e_0,\rho k_0)] = \bigcup_{\la e,k \ra \in \mc{W}_{g(e_0,k_0)}} B(\alpha e, \rho k)
	\]
	\end{Def}

	We next state Ceitin's theorem, also known as Kreisel-Lacombe-Shoenfield for metric spaces.

	\begin{Theorem}
	Let \((A^*,\nu_0\) and \((B^*,\nu_1)\) be computable metric spaces and \(f:\Ak \saa \Bk\) a \((\nu_0,\nu_1)\)-computable function. Then \(f\) is effectively continuous.
	\end{Theorem}
	In particular, \(f\) is continuous.

%% file: numberings.tex
\section{Effective Banach spaces}

	The axiomatization from \cite{Pour} gives us that the norm, effective limit and vector space operations shall be effective operations. 
	They start with a computable sequence and e.g.s. \(\{e_n\}\) for the separable Banach space \(\X\). 
	The computable elements are obtained as the effective closure of the rational span \(A = \{ a_n \}\) to \(\{e_n\}\), which is dense in \(\X\). 
	For recursion theoretical treatment it is better to express conditions using the concept of effective operation than using an axiomatization of the concept of computable sequence.

	\(A\) can be given a canonical totally defined numbering (surjective map) \(\alpha_0:\N \saa A\) where we use a standard numbering of the rational numbers \(\Q\), 
	whereas the effective closure \(\Ak = \Xk\) receives a partial numbering \(\alpha:\N \saa \Xk\).
	In general can one define the concept effective operator by

	\begin{Def}[\(\nu\)-effective operator]
		Let \(\nu\) be a numbering of a set \(X\). A operator \(F\) is called \(\nu\)-effective if \((\exists \varphi \in \mc{PR})(\forall e \in \op{dom}\nu)(F \nu e = \nu \varphi(e))\).
	\end{Def}

	The following follows from the definition:

	\begin{description}
		\item[  (i)] \(\op{dom} \nu \delm \op{dom} \varphi\)
		\item[ (ii)] \(\varphi[\op{dom} \nu] \delm \op{dom} \nu\)
		\item[(iii)] \(\varphi\) are extensional: \((\forall e,d \in \op{dom} \nu)(\nu e = \nu d \saa \nu f e = F \nu e = F \nu d = \nu f d)\)
	\end{description}

	With this in mind we define the concept of an {\em effective Banach space}.

	\begin{Def}
		Let \(\X\) be a separable Banach space and \(\nu:\N \saa \X\) a numbering.
		\((\X,\nu)\) is called an effective Banach space if the following operations are \(\nu\)-effective:

		\begin{description}
		\item[  (i)] vektoraddition \(+:\Xk^2 \saa \Xk\),
		\item[ (ii)] scalar multiplication with scalars in \(\Ck\) or \(\Rk\),
		\item[(iii)] the norm \(\| \cdot \|:\Xk \saa \Rk\) and
		\item[ (iv)] effective limits \(\lim^*:\XNk \saa \Xk\).
		\end{description}

		A computable element in \(\X\) is an element in \(\Xk \equiv \op{ran} \nu \delm \X\).
	\end{Def}
	Note that \(\Xk\) by construction is a countable subset of \(\X\).

	 We would like to make connections between the axiomatic approach in \cite{Pour} and the theory of effective operators. The basic observation is as follows.

	\begin{Theorem}[Sequential Effectivity] \label{Theorem:Sekvens}
	Let \(\X\) and \(\Y\) be effective Banach spaces.
	Each continuous function \(f:\X \saa \Y\) that preserves effective limes and computability of sequences are effective.
	\end{Theorem}

	\begin{proof}
	Let \(\{e_n\}\) be a e.g.s. in \(\X\). Let \(\{a_n\}\) be the rational span of \(\{e_n\}\), ordered in a computable sequence.  
	Since \(\{a_n\}\) is computable and a countable dense subset, a general computable element in \(x \in \Xk\) is an effective limit \(x= \lim^* a_{g(n)}\) (\(g \in \mc{R}\))  of elements from \(A\). 
	Since \(f\) are sequentially computable are \(\{f(a_n)\}\) a computable sequence. By continuity of \(f\) are \(f(x) = \lim f(a_{g(n)})\). 
	Modulus of convergence for \(x\) gives modulus of convergence for \(f(x)\) since \(f\) preserves effective limes.
	\end{proof}

	\begin{Kor}
	The axioms for computability structures gives effectivity of the relevant operations (linear combinations, effective limits and norm).
	\end{Kor}
	\begin{proof}
	For \Lin we choose \(\Y = \X\) or \(\Y = \X \times \X\) and get that \(+\), \(-\) and computable scalar multiplication are effective operations. For \Norm we choose \(\Y = \R\). 
	Norms preserve effective limits since \(| \|x|-\|y\| | \leq \|x-y\|\), so the norm is a effective operation.

	For \Lim we need a numbering \(\nu\) of \(\XNk\), and it is natural to let \(\nu e = \{ \alpha \varphi_e(n) \}_{n \in \N}\). 
	For lim\(^*\) to be effective then means that \(\lim^*_n \alpha \varphi_e(n) = \alpha f(e)\) where \(f \in \mc{R}\). 
	We define \(f\) implicitly through the following algorithm: for each \(n\) we have geometric convergence, and we have also geometric convergence in \(e\) for the limits in \(n\). 
	Pythagoras gives that the distance along the diagonal becomes \(\sqrt{2} \cdot 2^{-n} < 2^{-n+1}\) so a translation of the diagonal gives geometric convergence.
	\end{proof}

	The corollary gives that the axiomatic and ``numbering theoretic'' approaches to effective Banach spaces are equivalent. 
	We can therefore hereafter use the axioms for computability structures and effectivity of the various operations interchangeably.

\section{Numbering of \(\Xk\)}

	We present numberings of \(\Xk\), where \(\X\) is a separable Banach space.

	Given an e.g.s. \(\{ e_n \}\) we obtain a numbering from Effective Density, by putting together indices for modulus, the sum limit and the summand.
	In \cite{Pour} they start with computable sequences, but we will count individual elements. 
	We call a sequence \(\nu\)-computable if it can be written in the form \(\{ \nu f (e) \}_{e \in \N}\), where \(f \in \mc{R}\).

	We number the set \(A = \{ a_n \}\) of rational linear combinations from \(\{ e_n \}\), in such a way that we can recover \(\{ e_n \}\) as \(\{ a_{g(n)} \}\) where \(g \in \mc{R}\).

	The following numbering we call the standard numbering. In \cite{Blanck} the countable dense subset \(\{ a_n \}\) is fundamental whereas we in Banach spaces can generate \(\{ a_n \}\) from an e.g.s.

	\begin{Def}[Standard indices]

	Let \(e \in \op{dom}\alpha\) if:

	\begin{description}
	\item[ (i)] \(\varphi_e \in \mc{R}\)
	\item[(ii)] \((\forall n)(\| a_{\varphi_e(n)} - \lim a_{\varphi_e(\cdot)} \| < 2^{-n})\)
	\end{description}

	Let \((\alpha e)_n = a_{\varphi_e(n)}\) when \(e \in \op{dom} \alpha\).
	\end{Def}

	We call \(\alpha\) the standard numbering, and \(e \in \op{dom} \alpha\) are called standard indices.

	From a natural number \(e\) and the information \(e \in \op{dom} \alpha\) we can thus effectively find arbitrarily good approximations to \(\alpha e\). 
	This fits well with the intuition that e.g. the real number \(\pi = 3.14..\) is computable because given \(n\) we can effectively find the first \(n\) decimals.

	\paragraph{Complexity of dom \(\alpha\)}

	(ii) can be reformulated to a Cauchy-condition such that \(<\) compares rational numbers, that way (ii) becomes \(\Pi_1\). (i) is 
	\(\Pi_2\): \(\varphi_e \in \mc{R} \Hviss \forall n \exists s \varphi_{e,s}(n) \downarrow\). Thus dom \(\alpha\) is \(\Pi_2\).

	\paragraph{Complexity of equality}

	In the general domain theoretic case equality of computable elements is \(\Pi_2\). For example, equality of r.e. sets.
	But in the case of separable metric spaces equality is actually \(\Pi_1\): \(\alpha e = \alpha d \Hviss \| \alpha e - \alpha d \| = 0\) and equality of real numbers is \(\Pi_1\).

	\paragraph{Modulus of convergence}
	If we have a computable sequence \(\{x_n \}\) that has a subsequence \(\{x_{f(n)} \}\) (where \(f\) is recursive) that is geometrically convergent, 
	then we say that \(\{x_n \}\) is effectively convergent, and that \(f\) is a modulus of convergence for \(\{ x_n \}\).

\section{Numbering of \(\BXk\)}

	\begin{Def}[The numbering \(\tau\) of \(\BXk\)]
		\((\tau d)(\alpha e) = \alpha \varphi_d e\)
		if \(\varphi_d\) is
		\begin{itemize}
		\item \(\alpha e = \alpha c \saa \alpha \varphi_d e = \alpha \varphi_d c\) (\(\varphi_d\) are extensional)
		\item \(\varphi_d \in \mc{R}\)
		\item \(\varphi_d[\op{dom} \alpha] \delm \op{dom} \alpha\)
		\item \(\tau d\) becomes linear
		\item \(\tau d\) becomes bounded
		\end{itemize}

		Both extensionality and linearity are \(\Pi_1\) relative to equality \(\{ n,m | \alpha(n)=\alpha(m) \}\) and thus \(\Pi_1\).

		Boundedness, which is the same as continuity for linear operators, is \(\Sigma_2\):

		 \((\exists C)(\forall x)(\frac{\|Tx\|}{\|x\|} \leq C)\).
		The sets \(\{ e | \varphi_e \in \mc{R} \}\) and \(\op{dom} \alpha\) are both \(\Pi_2\).
	\end{Def}

	We introduce the notation \(T^e = \tau e\) to give some analysis intuition for the fact that the number \(e\) is assigned to an operator. 
	Then we may still use the notation \(\{ T_n \}\) for a sequence of operators, independent of numbering. 
	A sequence of effective operators can then be written as \(\{ T^{f(n)} \}\) and if \(f\) is recursive we call the sequence computable.

	\begin{Def}[The numbering \(\mu\) of \(\BXk\)]
		Let \({e_n = a_{g(n)}}\) be a e.g.s. and a computable sequence, i.e. \(g \in \mc{R}\).

		Let \(\mu c\) be the unique continuous extension \(T\) of \(e_n \mapsto \alpha \varphi_c(n)\), if
		\begin{itemize}
			\item \(\varphi_c \in \mc{R}\)
			\item \(T\) becomes linear
			\item \(T\) becomes bounded
		\end{itemize}
	\end{Def}

	Here the following concept is interesting.

	\begin{Def}[Schauder basis]
	A Schauder basis for a separable Banach space \(\X\) is a sequence \((x_n)\) such that for each \(x \in X\) there is a unique sequence \((\alpha_n)\) such that \(x = \sum_n \alpha_n x_n\).
	\end{Def}

	I.e. elements in a Schauder basis are linearly independent and have dense linear span.

	\begin{Prop}[\cite{Enflo}]
		There exist separable Banach space that have no Schauder basis.
	\end{Prop}

	If we limit ourselves to spaces with Schauder basis we can require that each e.g.s. be linearly independent, and then \(\mu c\) will automatically become linear. 
	We choose to stay within the framework of general Banach spaces. For general Banach space we have among other reasons the function space construction \(\X \mapsto \B(\X)\).

	 We would like to prove that \(\mu\) and \(\tau\) are counting the same set.
	For this we need some technical lemmas on effective limits.

	\begin{Lemma}
		Bounded operators preserves effective limits. \(T \lim^* = \lim^* T\).
	\end{Lemma}
	\begin{proof}
		\(T\) preserves \(\lim\) because \(T\) is bounded by the {\em First Main Theorem}. We must show that \(T\) preserves effectivity of convergence:
		Suppose \(x_n \rightarrow x\) effectively, i.e. there exists a recursive \(f\) such that
		\[
			n \geq f(N) \implies \| x_n - x \| < 2^{-N}
		\]
		Let \(h(N)=f(K+N-1)\), where \(2^{-K} < \frac{1}{\| T \|} < 2^{-(K-1)}\). Then have vi
		\[
			\| Tx_n - Tx \| = \|T(x_n-x)\| \leq \| T \| \| x_n - x \| < 2^{-N}
		\]
		i.e. \(Tx_n \rightarrow Tx\) effectively.
	\end{proof}

	\begin{Lemma}
		\[
		 	\lim^{*}(\alpha x_n) = \alpha \lim^{*} x_n
		\]
	\end{Lemma}
	\begin{proof}
		As above, but choose \(K\) such that \(2^{-K} < \frac{1}{\alpha} < 2^{-(K-1)}\).
	\end{proof}

	\begin{Lemma} 
		\(\lim^{*}(x_n + y_n) = \lim^{*}x_n + \lim^{*}y_n\) if the right hand side exists.
	\end{Lemma}
	\begin{proof}
		Assume the right hand side exists, i.e. there exist recursive \(f\) and \(g\) such that
		\[
			n \geq f(N) \implies \| x_n - x \| < 2^{-N}
		\]
		\[
			n \geq g(N) \implies \| y_n - y \| < 2^{-N}
		\]
		Let \(h(N)=\op{max}\{f(N+1), g(N+1)\}\). Then we have
		\[
			n \geq h(N) \implies \|(x_n +y_n) - (x+y)\| \leq \|x_n-x\| + \|y_n-y\| < 2^{-(N+1)} + 2^{-(N+1)} = 2^{-N}
		\]
	\end{proof}
	
	\begin{Lemma}
		\( \lim^{*}_k \lim^{*}_l x_{kl} = \lim^{*}_k x_{k,f(k)} \), where \(f(k)=\mu l (|x_{kl} - x_k | < 2^{-k}\).
	\end{Lemma}
	\begin{proof}
		Suppose
		\[
			k \geq g(N) \implies \|x_k - x\| < 2^{-N}
		\]
		Let \(h(N)=g(N+1)\). We may assume \(g(N) \geq N\). It follows that
		\[
		 	k \geq h(N) \implies \|x_{k,f(k)}-x\| \leq \|x_{k,f(k)} - x_k\| + \|x_k - x\| < 2^{-k} + 2^{-(N+1)} \leq 2^{-N}
		\]
		since
		\[
		 	k \geq h(N) = g(N+1) \geq N+1
		\]
	\end{proof}
	
	\begin{Lemma}
		\[
			\sum^x_{j=0} \sum^{f(j)}_{i=0} a_{ij} e_i = \sum^{\op{max}_{j=0..x} f(j)}_{i=0} \sum^x_{j=0} a_{ij} \Theta(f(j)-i) e_i
		\]
		where \(\Theta(x)\) (the Heaviside function) is \(1\) for \(x \geq 0\) and \(0\) otherwise.
	\end{Lemma}
	This is easily seen, e.g. by drawing a diagram.

	Let \(C\) and \(D\) be subsets of \(\N\). Two partial numberings \(\nu_0:C \saa \Xk\) and \(\nu_1:D \saa \Xk\) are called equivalent if \(\nu_0 \leq \nu_1\) and \(\nu_1 \leq \nu_0\), 
	where the order is given by \(\nu_0 \leq \nu_1 \Hviss (\exists \varphi \in \mc{PR})((\varphi[C] \delm D \snitt \op{dom} \varphi) \og (\forall e \in C)(\nu_0 e = \nu_1 \varphi e))\).

	\begin{Theorem}
	\(\op{ran} \mu = \op{ran} \tau\), and \(\mu\) and \(\tau\) are equivalent numberings of \(\BXk\).
	\end{Theorem}
	\begin{proof}
		Suppose that \(T \in \op{ran} \mu\) and that \(x\) is a computable element. According to Effective Density we may write
		\[
		 	x = \lim^{*}_k \sum^{d(k)}_{j=0} \alpha_{kj} e_j
		\]
		where \(d\) is recursive, \(\alpha\) is a computable double sequence of rational numbers, and \(\{ e_n \}\) is an e.g.s.
		Furthermore \({T e_j}\) is computable, and we may write
		\[
			T e_j =  \lim^{*}_l \sum^{c(j,l)}_{i=0} \beta_{jli} e_i
		\]
		by Effective Density.

		We get an explicit expression that shows that \(T \in \op{ran} \tau\) by the above lemmas and the following calculation:
		\[
			Tx = T \lim^*_k \sum^{d(k)}_{j=0} \alpha_{kj} e_j = \lim^*_k \sum^{d(k)}_{j=0} \alpha_{kj} Te_j =
			\lim^*_k \sum^{d(k)}_{j=0} \alpha_{kj} \lim^*_l
			\sum^{c(j,l)}_{i=0} 
			\beta_{jli}	e_i	=
		\]
		\[
			\lim^*_k \lim^*_l \sum^{d(k)}_{j=0} \sum^{c(j,l)}_{i=0} \alpha_{kj} \beta_{jli} e_i 
			= \lim^*_k \sum^{d(k)}_{j=0} \sum^{c(j,f(k))}_{i=0} \alpha_{kj} \beta_{j,f(k),i} e_i = 
		\]
		\[
		 	\lim^*_k \sum^{d^{'}(k)}_{j=0}  \alpha_{k,(j)_1} \beta_{(j)_1,f(k),(j)_0}  
			\Theta(c((j)_1,f(k)) - (j)_0 )  \Theta(d(k) - (j)_1) e_{(j)_0} = 
		\]
		\[
		 	(d^{'}(k) = \op{max}_{j=0..d(k)} \op{max} \{ c(j,f(k)), d(k) \}^2) 
		\]
		\[
		 	\lim^{*}_k \sum^{d^{'}(k)}_{j=0} \alpha^{'}_{kj} e_{(j)_0} 
			= \lim^{*}_k \sum^{d^{'}(k)}_{j=0} \alpha^{'}_{kj} \sum^{(j)_0}_{l=0} \delta_{l,(j)_0} e_l 
			= \lim^{*}_k \sum^{d^{''}(k)}_{j=0} \alpha^{''}_{kj} e_j 
		\]
		for suitable \(d^{''}\), \(\alpha^{'}\), and \(\alpha^{''}\), so \(T\) is effective in standard meaning, i.e. 
		\(T \in \op{ran} \tau\). Since \(d^{''}\) and \(\alpha^{''}\) are found effectively from \(\alpha\) and \(d\) we also have \(\mu \leq \tau\).

		Now suppose that \(T = \tau c\). Since \(\{e_n\} = \{ a_{g(n)} = \alpha h n \}\) (\(\exists h \in \mc{R}\)) is a computable sequence, we have \(\{Te_n = \alpha \varphi_c h n \}\). 
		Define \(d\) by \(\varphi_d = \varphi_c h\). Then \(T = \mu d\), and in particular \(\tau \leq \mu\).
	\end{proof}

\section{Quasi-effective operators}

	\begin{Def}[Quasi-indices]
		A quasi-index is a \(\beta\)-index, where the numbering \(\beta\) is defined as follows.

		Let \((\beta \la e,k \ra )_n = a_0\) if \(\varphi_{e,n}(0) \uparrow\).
		Otherwise, let \((\beta \la e,k \ra )_n = a_{\varphi_{e,n}(0)}\) if the following holds:
		\begin{description}
			\item[ (i)] \(\varphi_{e,n}(0) \downarrow\)
			\item[(ii)] \(((n=\mu n(\varphi_{e,n}(0) \downarrow)) \, \eller \, \varphi_{e,n}(1) \uparrow)\)
		\end{description}

		Otherwise, let \((\beta \la e,k \ra )_n = a_{\varphi_{e,n}(m)}\) if the following holds:

		\begin{description}
			\item[  (i)] \(\varphi_{e,n}(m) \downarrow\)
			\item[ (ii)] \(\|a_{\varphi_e(m)}-a_{\varphi_e(m-1)}\| < 2^{-m}k\)
			\item[(iii)] \((n = \mu n (\varphi_{e,n}(m) \downarrow) \eller \varphi_{e,n}(m+1) \uparrow \eller \|a_{\varphi_e(m+1)}-a_{\varphi_e(m)}\| > k \cdot 2^{-(m+1)})\)
		\end{description}

	\end{Def}

	\(F\) is called {\em effective} if it is \(\alpha\)-effective, and {\em quasi-effective} if it is \(\beta\)-effective.

	Intuition: Using quasi-indices we wish to find a totally defined numbering of the computable elements in \(\X\), 
	such that we for each \(e\) and \(k\) can know that \(\beta \la e,k \ra\) is a computable real number, 
	and effectively find a sequence \((\beta e)_n\) that converges geometrically (up to a factor \(k\)) to \(\beta e\). 
	It appears that \(\beta\) is among the most natural numberings we then can choose; if \(e \in \op{dom} \alpha\) we get \(\alpha e = \beta e\). 
	If \(\varphi_e\) is not totally defined we get a problem, and the only way to solve it on appears to be to let the computation of \(\varphi_e(n)\) run indefinitely, 
	so that the sequence \((\beta e)\) becomes constant starting when we begin to attempt to calculate a value of \(\varphi_e\) where it is undefined. 
	Also note that if \(\varphi_e\) is totally undefined (which is quite common) we get the sequence \(a_0, a_0, ..\). 
	In place of \(a_0\) we could have used an arbitrary element of \(A\).

	Let the set of all {\em quasi-sequences} be \(\beta[\N]\).

	\begin{Lemma}
		\label{Lemma:konstant}
		Constant subsequences of a quasi-sequence having limit \(\notin A\) have finite length.
	\end{Lemma}

	\begin{proof}
		The elements of a quasi-sequence are in \(A\). Sequences with constant subsequences of infinite length are constant except for a finite initial segment, 
		and therefore have limit in \(A\).
	\end{proof}

	\begin{Theorem}
		For each effective Banach space (that is not zero-dimensional) there exist bounded linear operators that is not quasi-effective.
		 \label{Thm:sqrt}
	\end{Theorem}
	\begin{proof}
		For \(x \in \Xk\), let \((x)_n\) be a quasi-sequence with limit \(x\).

		Choose \(e\) such that \(\varphi_e\) is calculated by the following algorithm:
		\[
			a_{\varphi_e(x,n)} = 
			\begin{Cases}
				a_0 & n=0 \\
				2a_0 \cdot \frac{\varphi_x(x)}{\varphi_x(x)} & n \geq 1 \\
			\end{Cases}
		\]

		Let \(\zeta^x = \beta \la S^1_1(e,x), 1 \ra\), such that \(\zeta^x\) is a quasi-sequence uniformly in \(x\).
		The it follows that
		\[
			\lim \zeta^x = \begin{Cases}
		                  a_0  & \varphi_x(x) \uparrow \\
		                  2a_0 & \varphi_x(x) \downarrow \\
		                  \end{Cases}
		\]

		Let \(F\) be the effective operator \(x \mapsto \sqrt{2}x\). Suppose that \(F\) is quasi-effective.

		Since \(F \lim \zeta^x \in \{ \sqrt{2} a_0, 2 \sqrt{2} a_0 \} \delm \Xk \setminus A\), by Lemma~\ref{Lemma:konstant} we can choose \(n\) so large that the following holds:
		\begin{description}
		\item[ (i)] \(\|(\sqrt{2} a_0)_n - (2 \sqrt{2} a_0)_n       \| > 2^{-(n-1)}\)
		\item[(ii)] \(\|(F \lim \zeta^x)_n - F \lim \zeta^x         \| < 2^{-n}\)
		\end{description}

		But then we can solve the halting problem, since
		\[
			\|(F \lim \zeta^x)_n - (\sqrt{2} a_0 )_n\| < 2^{-n} \Hviss \varphi_x(x) \uparrow
		\]
	\end{proof}

	Linearity has a decisive influence on existence of effective operators that is not quasi-effective. A non-linear example on \(\Rk\) is \(x \mapsto x + \sqrt{2}\).

	Lemma ~\ref{Lemma:konstant} also holds for \(\alpha\)-indices.

	The decisive difference between \(\alpha\) and \(\beta\) is that we cannot find an \(\alpha\)-index for \(\zeta^x\) uniformly in \(x\).

	\paragraph{Quasi-effectivity and r.e.-indices}
	Quasi-numbering of \(\Xk\) corresponds to a numbering of the recursive sets using r.e.-indices.

	Standard-numbering of \(\Xk\) corresponds to a numbering of the recursive sets using characteristic indices, 
	i.e. indices for characteristic functions \(c_A\), \(A\) recursive.
	In this way we can also express the intuition that a totally undefined function \(\varphi_e\) is a kind of computable object: 
	we can effectively in \(n\) decide what \(\varphi_e(n)\) is (namely undefined, ``\(\uparrow\)''). 
	I.e. that the graph of \(\varphi_e\), \(G_{\varphi_e} =\emptyset\). Thus \(c_{G_{\varphi_e}}\), the characteristic function of the graph, is recursive.

	\begin{Lemma}
	There exists a total computable numbering of the recursive sets in following sense: 
	There exists a recursive function \(f\) such that for all \(e\), \(f(e)\) is an r.e.-index for a recursive set, and all recursive sets have such an index.
	\end{Lemma}

	\begin{Lemma}
		There exists no total computable numbering of the recursive sets in following sense: 
		There exists no recursive function \(f\) such that for all \(e\), \(f(e)\) is a characteristic index for a recursive set, and all recursive sets have such an index.
	\end{Lemma}
	\begin{proof}
		If \(f\) is recursive and \(\{ \varphi_{f(e)} \}_{e \in \N}\) is a list of recursive functions, then \(x \mapsto \varphi_x(x)+1\) is recursive, 
		but not equal to \(\varphi_{f(e)}\) on input \(e\), and thus not on the list.
	\end{proof}

	These to lemmas express a parallel with the relationship between quasi-numbering and standard-numbering of \(\Rk\).

\section{Total numberings do not give modulus}
	Other things being equal it would be desirable to have totally defined numberings. It is essentially the same as to say that the domain of definition of the numbering is recursive. 
	In the theory of numberings one often confines attention to totally defined numberings.

	\begin{Lemma}
		Let \(x = \lim_n q_{f(n)}\) (\(f \in \mc{R}\)), where \(\{q_{f(n)}\}\) is a geometrically convergent sequence of rational numbers. We can effectively find \(y\) such that
		\[
		 	x \equiv \sum_{n=1}^\infty x_n 3^{-n} (\op{mod} 1) \og y \equiv \sum_{n=1}^\infty y_n 3^{-n} (\op{mod} 1) \saa (\forall k)(x_k \neq y_k) 
		\]
	\end{Lemma}
	\begin{proof}
		Since \(\{ q_{f(n)} \}\) is geometrically convergent we can effectively determine \(x\) with arbitrary precision, 
		and thus also with large enough precision to exclude a coefficient \(y_k \in \{ 0,1,2\}\) for each \(k\).
	\end{proof}

	\begin{Theorem}
		Let \(X\) be a effective Banach space that is not zero-dimensional.
		Total numberings of \(\Xk\) do not give modulus.
	\end{Theorem}
	\begin{proof}
		Let \(\{ a_{ij} \}_{i,j \in \N}\) be a matrix of elements of the countable dense subset \(A=\{ a_n \}\). 
		Suppose the matrix represents a numbering of \(\Xk\) that gives modulus: 
		\(\forall i,j \quad \|a_{i,j} - a_{i,j+1}\| < 2^{-j}\), and in particular that the matrix is a computable sequence. 
		We shall show that this ``numbering'' is not surjective.
		Since \(\X\) is an effective Banach space we can effectively form the matrix of norms \(\{ \| a_{ij} \| \}\).
		We can then form the matrix \(\{ a'_{ij} \}\) where \(a'_{ij} \delm \{ 0,1,2 \}\) as follows: 
		Write \(\| a_{ij} \| \op{mod} 1\) as a power series \(\sum_{k=1}^\infty c_{ijk} 3^{-k}\), where \(c_{ijk} \in \{ 0,1,2\}\). 
		We can effectively find a \(d_{ijk} \in \{0,1,2\}\) such that \(d_{ijk} \neq c_{ijk}\).

		Let \(b = \sum_{k=1}^\infty d_{kkk} 3^{-k} (\in [0,1])\).

		Then we have \(\neg \exists in \quad b =\| \lim_j a_{ij}\|\).

		Since \(\X\) is not zero-dimensional, and \(\Xk\) is dense in \(\X\), there exists a \(x \in \Xk\), \(x \neq 0\). 
		Then \(y=\frac{b}{\|x\|}x \in \Xk\), and \(\|y\|=b\), but \(\forall in \quad y \neq \lim_j a_{ij}\).
	\end{proof}

	We have here seen an example of the phenomenon that phenomena from pure recursion theory often transfer to computable analysis:
	\begin{itemize}
		\item
			\((\neg \exists f \in \mc{R})(\mc{R}=\{\varphi_{f(n)} | n \in \N \})\)
		\item
			\((\neg \exists f \in \mc{R})(\Xk   =\{\lim^* a_{f(n)}| n \in \N \})\)
	\end{itemize}

	Quasi-effectivity seems to be a concept that falls between two chairs in the following sense: 
	\(x \mapsto \sqrt{2} x\) is not quasi-effective, because quasi-indices, that do not generally give modulus,
	do give modulus for elements outside the original countable dense subset.
	For a general total numbering that does not at all give modulus, the argument from the quasi-effective case cannot be used, 
	so we need a new argument to answer the question:

	Is there a total numbering \(\nu\) of effective Banach spaces that gives the same class of effective operators as the standard-numbering \(\alpha\)?

%% file: effective-operators.tex
\section{The Kreisel-Lacombe-Shoenfield theorem}

	\paragraph{KLS for effective Banach space}
	We use the word ``function'' to emphasize that a result is not restricted to linear operators.

	\begin{Theorem}[Ceitin-Blanck]
		Let \(f\) be a \(\-{\alpha}\)-computable function on a computable metric space \(A\). Then \(f\) is effectively continuous.
	\end{Theorem}

	\begin{Theorem}
		Effective Banach spaces \(\X\) are effective metric spaces (as in \cite{Blanck}).
	\end{Theorem}
	\begin{proof}
		Since \(\X\) is effective there exists an e.g.s. \(\{e_n\}\).

		Let \(A = \{a_n\}\) be the rational span of \(\{e_n\}\) equipped with metric \(d(x,y)=\|x-y\|\).
		\(A\) can be made into a computable sequence in \(\X\).

		By the Norm axiom \(\{ \| a_n \| \}\) is a computable sequence of real numbers.

		By Effective Density a general element in \(\Xk\) is of the form \(x=\lim^*_n a_{f(n)}\), where \(f\) are recursive.

		By the composition property the norm \(\|x\| = \| \lim^*_n a_{f(n)} \| = \lim^*_n \| a_{f(n)} \|\) (with same modulus of convergence) is effective.

		\((x,y) \mapsto (x - y)\) is effective by \Lin and Sequential Effectivity.
		The effective operations are closed under composition, so the metric \(d(x,y)=\|x-y\|\) is effective.

		Thus \(A\) is a computable metric space, and the closure an effective metric space.
	\end{proof}

	\begin{Lemma}
		Let \(T\) be a linear transformation between to normed vector spaces. Then the following are equivalent:
		\begin{description}
			\item[  (i)] \(T\) is continuous in a point \(x_0\)
			\item[ (ii)] \(T\) is uniformly continuous
			\item[(iii)] \(T\) is bounded
		\end{description}
	\end{Lemma}
%
%

	\begin{Lemma}[Bounded linear transformations]
		If \(\X_0\) is dense in \(\X\), then each operator \(T\) in \(\B(\X_0,\X)\) has a unique extension to an operator \(\tilde{T}\) in \(\B(\X)\).
	\end{Lemma}

	\begin{Theorem}[KLS for effective Banach spaces]
		Let \(\X\) be an effective Banach space.
		Then each effective function defined on \(\Xk\) is continuous.
	\end{Theorem}
	\begin{proof}
		Since effective Banach space is an effective metric space this follows from Ceitin's theorem.
	 \end{proof}

	\begin{Theorem}[Strong version of KLS] \label{Theorem:Strong KLS}
		Let \(\X\) be a separable Banach space and \(T\) an effective linear operator on \(\Xk\). Then \(T\) has a unique continuous linear extension to all of \(\X\).
	\end{Theorem}
	\begin{proof}
		Choose an e.g.s. \(\{e_n\}\). Then \(\X\) is the closure of the linear span of \(\{ e_n\}\).
		Since \(T\) is effective, \(T\) is defined for \(\{e_n\}\). We extend therefore \(T\) uniquely by linearity on the span of \(\{e_n\}\), and thereafter uniquely by continuity on the closure.
	\end{proof}

	This is an example of the recursion theoretical significance of linearity. 
	In mathematics and perhaps especially in recursion theory we wish to reduce the infinite to the finite, 
	and the uncountable to the countable. By linearity we have above so to say {\em reduced computability to effectivity}; 
	by computability we mean then the stronger requirement to act computably even on non-computable input.

	Effective operators can by the above be considered to be totally defined. In short: If \(T\) is effective, then \(T\) is continuous on all of \(\X\).

	Bounded operators are often required to be totally defined, but then for conventional rather than theoretical reasons. 
	Continuity is for example preserved when restricting to closed subspaces. 
	We have thus found a theoretical reason for a convention: The continuous operators that arise in analysis are as a rule effective. 
	This can be given a philosophical treatment based on Church's thesis, see for example section I.9 in \cite{Odi}.

	Furthermore it is clear that linearity of an operator is a good property in several ways: 
	on \(n\)-dimensional Banach spaces they can be described using \(n^2\) elements of the scalar field \(\mb{F} \in \{ \R, \C \}\), 
	whereas continuous non-linear operators are described using using countably many elements from \(\mb{F}\), 
	since a continuous operator is determined by its action on a dense subset. 
	And non-continuous non-linear operators must be described by uncountably many (\(\beth_1\), i.e. ``continuous many'') elements from \(\mb{F}\).

	\paragraph{Effective continuity}
	In the original KLS-theorem for operators on the space of partial functions \(\N \saa \N\), \(\mc{P}\), 
	an operator is effective iff it has an effectively continuous extension. 
	But a bounded linear operator is effectively uniformly continuous, so for Banach spaces (and thus for metric spaces) ``if and only if'' must be replaced by ``only if''. 
	The generalization of KLS to domain theory has also just ``only if'', i.e. effectivity implies continuity. 
	The concept of effective continuity stands in a position between effectivity (a relatively strong requirement) and continuity (a relatively weak requirement).

	KLS gives that effective operators are continuous on \(\Xk\).

	It is known that there exist effective functions on \(\Rk\) that have no continuous extension to all of \(\R\):

	Define the triangle function \(\tau(x_1,x_2,h)(x)=\op{max} \{0, h-|x-(\frac{x_1-x_2}{2})| \}\) with support on the interval \([x_1,x_2]\) and height \(h\).

	\begin{Lemma}
		There exists an effective function on \(\Xk\) (\(\X\) an effective Banach space) that is not the restriction of any continuous function on \(\X\).
	\end{Lemma}
	\begin{proof}
		We show that there exists an effective function \(f:\Rk \saa \Rk\) that (by KLS) is continuous on \(\Rk\), 
		but does not have a continuous extension to all of \(\R\).

		Let \(T \subset \{ 1,2 \}^{\N}\) be a recursive tree with no recursive branch.

		Let \(f = \sum_{n=1}^{\infty} 
		\sum_{\sigma \in \{1,2\}^n \snitt T} \tau(\sum_{i=1}^n \sigma_i 4^{-i}, \sum_{i=1}^n \sigma_i 4^{-i} +4^{-n},2^{-n})\).

		In 4-ary representation we here use the digits 1 and 2, but not 0 or 3, since \(0.1333..=0.2\) and we do not want discontinuity in computable pointer such as 0.2.
		\(f\) is effective on \(\Rk\) but goes to infinity near some points in \(\R \setminus \Rk\), and thus has no continuous extension to \(\R\).
	\end{proof}

	Since uniformly continuous functions on a dense subset of a metric space have a continuous extension to all of the space, 
	there thus exist effective functions that are not uniformly continuous. 
	It is therefore natural as in \cite{Pour} to define computable functions as those that are effective and effectively uniformly continuous.

	\begin{Prop}
		There exist total discontinuous linear operators on separable Banach space.
	\end{Prop}
%

	\paragraph{Closed operators}
	Most operators arising in mathematical physics are closed. For a discussion see \cite{Pour}.

	\begin{Theorem}[Closed operators]
		Each effective operator is closed on \(\Xk\), and closed on \(\X\) if it is totally defined. There exist effective non-closed operators, and they are not totally defined.
	\end{Theorem}
	\begin{proof}
		Suppose \(x_n \saa x \in \Xk, Tx_n \saa y\). Since \(T\) by KLS is continuous on \(\Xk\), we have \(x \in \op{dom} T\), and \(Tx_n \saa Tx\). Since \(\X\) is Hausdorff, limits are unique and \(Tx=y\).

		The effective operator \(T = \op{id}|_{\op{span}(\Xk)}\) is not closed on \(\X\). For generally id on a proper dense subspace \(Z\) of \(\X\) cannot be closed. 
		Let namely \(x \notin Z\) and choose \(x_n \in Z, x_n \saa x\). Then we have \(Tx_n = x_n \saa x\), but \(x \notin \op{dom}(T)\). 
		But such counterexamples will have a unique closed extension by Theorem ~\ref{Theorem:Strong KLS}. 
		If we as suggested consider each effective operator to be total, then all effective operators become closed. 
		But this consideration is not necessary for Theorem ~\ref{Theorem:Strong KLS} itself, so we have a minor asymmetry between continuous and closed operators.
	\end{proof}

\section{Characterization of effective operators}

We can characterize the effective linear operators as follows: The effective operators are the continuous operators that act effectively on an e.g.s.

By the First Main Theorem we know the following: Let \(\X\) be an effective Banach space and \(\{ e_n \}\) an e.g.s. 
If \(T\) is closed, bounded and effectively determined then \(T\) is sequentially effective, i.e. \(T\) sends computable sequences to computable sequences.

\begin{Theorem}
Let \(\X\) be an effective Banach space. A linear operator \(T\) defined on \(\Xk\) is effective iff it is effectively determined and continuous.
\end{Theorem}

\begin{proof}
Suppose \(T\) is effective. Then \(T\) is continuous by KLS, and effectively determined since \(\{ e_n \}\) is a computable sequence.

Now suppose that \(T\) is effectively determined and continuous.
\(T\) can be considered to be closed since it has a unique continuous extension to all of \(\X\), and each totally defined continuous operator is closed.
Now the First Main Theorem gives us that \(T\) sends computable sequences to computable sequences, i.e. \(T\) is sequentially effective. By Sequential Effectivity \(T\) is effective.
\end{proof}

%% file: operator-spaces.tex
One can put several norms on sets of linear operators on Banach space, 
but it is the operator norm \(\|T\| = \sup_{\|x\|=1} \|Tx\|\) that gives that the bounded operators are the continuous ones.

It is with this norm that the space of operators on a Banach space becomes another Banach space. 
The operator norm is not computable for operators on infinite dimensional Banach space. 
The closes we get is the effectively closed effective ideal of effectively compact operators, 
where the operator norm becomes an effective operation in the case that the Banach space has the approximation property.

\section{Counterexamples for the operator norm}
	\input{counterexamples.tex}

\section{Examples of effective operator norm}
	\subsection{\(\B(\C^n)\)}
		Banach space of dimension \(n\) are Hilbert space. They are moreover all homeomorphic and isomorphic with \(\mb{F}^n\). 
		Line{\ae}re operators on these spaces are therefore represented by \(n \times n\)-matrices.

		For operators on finite dimensional Hilbert space we have

		\[ \| A \| = \op{max} \{ \sqrt{|\lambda|} | (\exists x \neq 0)(  A^* A x = \lambda x ) \} \]

		To find the norm it is thus enough to find they finite many eigenvalues to \(A^* A\) 
		and thereafter pick one that in absolute value is greater or equal all the others.
		One can effectively find eigenvalues \(\lambda\) to a matrix \(B\) by to solve the \(n\)th degree equation \(\op{det}(\lambda I - B)=0\). 
		This can be done effectively by the effective version of the fundamental theorem of algebra:

		\begin{Theorem}[Effective version of the fundamental theorem of algebra]
			\label{th:EFA}
			La
			\[
				p(x)=\sum_{k=1}^n \alpha_k x^k
			\]
			 be a polynomial with computable complex coefficients \(\{ \alpha_k \}_{k=1}^n\). 
			Then there exists \(n\) computable complex numbers \(\{ \beta_k \}_{k=1}^n\) such that
			\[
				p(x) = \prod_{k=1}^n (x-\beta_k)
			\]
			\(\{ \beta_k \}_{k=1}^n\) can be found effectively from \(\{ \alpha_k \}_{k=1}^n\).
		\end{Theorem}

		Proof sketch: 
			We can effectively find a compact interval \(K\) where all the zeros must be located. 
			By drawing the graph of the polynomial on \(K\) in a grid with progressively smaller squares, 
			we can effectively approximate the zeros.

		\begin{Prop}
			There exists a computability structure on \(\B(\C^n)\) where the computable elements are 
			the matrices whose matrix elements are computable scalars.
		\end{Prop}
		\begin{proof} 
			We have seen that the norm is effective on these. 
			The set of \(n \times n\)-matrices with computable scalars are clearly closed under linear combinations with computable scalars. 
			If the matrices \(\{ A_n \}\) converges effectively to \(A\) in norm, 
			then also each individual matrix element converges effectively, so \(A\) must also have computable scalars. 
		\end{proof}

		The norm in \(\B(\C^n)\) are effective also if the operator are given by its graph, 
		for from the graph we can effectively find the matrix elements 
		\(a_{mn} = (e_n|Te_m)\) where \(\{ e_n \}\) is a standard orthonormal basis.

	\subsection{\(\B_f(\X)\) and \(\B_0(\X)\)}

		Let \(\B_f(\X)\) be the class of operators \(T \in \B(\X)\) with finite rank, i.e. such that the dimension to \(T[\X]\) is finite. 
		Note that when \(\X\) is infinite dimensional these can, in contradistinction to operators in \(\B(\C^n)\), 
		not be described by finitely many scalars, since a vector in \(\X\) generally cannot be described by finitely many scalars.

		\(\B_f(\X)\) are not closed in \(\B(\X)\) and therefore not a Banach space.

		The set \(\B_0(\X) \delm \B(\X)\) of {\em compact} operators best{\aa}r of the \(T\) such that 
		for each bounded sequence \(\{ x_n \}\) have the sequence \(\{ T x_n \}\) a accumulation point.
		\(\B_0(\X)\) is a separable closed subspace of \(\B(\X)\), and thus a potentially effective Banach space.

		A Banach space \(\X\) is said to have {\em the approximation property} if it for each compact operator \(K\) there is a sequence 
		\(\{ K_n \}\) of operators of finite rank such that \(\| K_n - K \| \saa 0\).

		Separable Hilbert space \(\H\) have the approximation property, but 
		it is not known whether \(\B(\H)\) has the approximation property (se \cite{Ban}). 
		The approximation property is a property of Banach space, not a topological property. 
		It is also essentially a Banach space property of \(\H\) and not of \(\B(\H)\), 
		for if we consider \(\B(\H)\) just as a Banach space we look at elements just as vectors and not as operators.

		Per Enflo proved that not all separable Banach space have the approximation property, 
		with corollary that not all separable Banach space have a Schauder basis.

		Effective normal compact operators \(T\) have computable norm (\(\|T\| \in \Rk\)). 
		This follows since the norm is the absolute value of one of the eigenvalues, 
		and all eigenvalues of a normal operator are computable according to the following theorem from \cite{Pour}.

		\begin{Theorem}[Pour-El \& Richards Second Main Theorem]
			Let \(T\) be a normal operator on an effective Hilbert space. 
			Then all eigenvalues to \(T\) computable, but the sequence of eigenvalues need not be.
		\end{Theorem}

		If we on Banach space \(\X\) with the approximation property limit ourselves 
		to they effective compact operators that are {\em effectively compact}, 
		effective limits for finite rank operators, becomes the operator norm an effective operation:

		Let \((x \odot y)(z)=(z|y)x\), where \(x,y \in \H\). Then the set of rank \(1\) operators equals \(\{ x \odot y | x,y \in \H \}\).

		The set of rank \(n\) operators are \(\{ \sum_{i=1}^n x_i \odot y_i | x_i, y_i \in \H \}\).

		The set of finite rank operators are \(\B_f(\H)= \{ \sum_{i=1}^n x_i \odot y_i | n \in \N, x_i, y_i \in \H \}\).

		The set of compact operators, the closure of the set of finite rank operators, are 
		\(\B_0(\H) = \{ \lim_{n \saa \infty} \sum_{i=1}^{d(n)} x_{i,n} \odot y_{i,n} | x_{i,n}, y_{i,n} \in \H \}\).

		\begin{Theorem}
			Let \(\X\) be a Banach space with the approximation property.
			Then the norm an effective operation for effectively compact operators on \(\X\).
		\end{Theorem}

		\begin{proof}
			By to limit ourselves to computable double sequences \((x_{i,n})\) and \((y_{i,n})\), recursive \(d\) and effective limes, 
			we get the set of {\em effectively compact operators}. 
			These are clearly compact, and effective since norm convergence implies point-wise convergence.

			For rank 1 operators we have \(\| x \odot y \| = \|x\| \|y\|\). 
			For rank \(n\) operators we get a complicated expression that nevertheless is computable. 
			Normen to an effective limes of finite rank operators are effective limes of the norm to finite rank operators.
		\end{proof}

		In general we can say that 
		the concept of effective Banach space advantageously could be limited to separable Banach spaces with Schauder bases. 
		Then the numbering \(\mu\) would get a lower complexity, and \(\B_0(\X)\) would have a computable norm since 
		each Banach space with Schauder bases has the approximation property, se \cite{Ban}. 
		All separable Banach space that naturally arise in analysis have Schauder bases, 
		the counterexamples by Per Enflo was {\em constructed} (but nevertheless a counterexample).

		Banach space \(\X\) with Schauder bases are separable, 
		then the rational linear span of a Schauder basis is a countable dense subset in \(\X\).
		Separable Hilbert space have in particular good Schauder bases, namely countable orthonormal bases.

		\begin{Bem}
			There exist \(\beth_1\) (i.e. continuum many) orthonormal bases for each separable Hilbert space of dimension \(>1\).
		\end{Bem}

		\begin{proof}
			We can find \(\beth_1\) unit vectors that are mutually non-orthogonal already in \(\F^2\), namely 
			\(\{ (\op{cos} \theta, \op{sin} \theta) | 0 \leq \theta < 2 \pi \}\), 
			and \(\F^2\) can be embedded in each \(>1\)-dimensional Hilbert space. 
			By Gram-Schmidt's orthogonalization process these unit vectors are incorporated in one orthonormal basis each.
		\end{proof}

		Av disse \(\beth_1\) orthonormal bases choose we ut en, and call it computable. 
		This are implicitly in functional analysis when one choose a orthonormal basis that are ``easy to calculate with'', 
		such as e.g. \(\{ t \mapsto \frac{1}{\sqrt{2 \pi}} e^{int} \}_{n \in \N}\) for \(L^2([-\pi,\pi])\).

		\begin{Bem}
			If the separable Banach space \(\X\) and its dual \(\X^*\) both have Schauder bases, 
			then elements of \(\B(\X)\) can be given by scalar matrices of countable dimension.
		\end{Bem}

		\begin{proof}
			Let namely \((e_n)\) be a Schauder basis for \(\X\). 
			I separable spaces are this countable, since we where cannot ha uncountably many independent vectors: 
			each of these vectors will namely be contained in an open set that does not contain any of the other vectors in the basis.

			For \(T \in \B(\X)\) we have by continuity \(Tx = T \sum \alpha_n e_n = \sum \alpha_n (T e_n)\) and \(T e_m = \sum \beta_{mn} e_n\), 
			so we can consider \(T\) to be the matrix \((\beta_{mn})\).
		\end{proof}

		Since Schauder bases have dense span it is wise to choose a Schauder basis as e.g.s. (when one exists).

		\paragraph{Representation of mathematical objects}

		How we imagine a mathematical object as given, what questions we can pose to an oracle about the object, 
		has significance for what operations can be considered effective.
		Can we effectively find the norm of a linear operator? This depends on how we are given operator: 
		by the graph \(\{  \la x,Tx \ra | x \in \X \}\), 
		or (in the Hilbert space case) by the matrix elements \(\{ a_{ij} = (e_i|Te_j) | i,j \in \N \}\) 
		where \(\{ e_j \}\) is a orthonormal basis.

		One has not found any simple condition on the matrix elements for an operator to be bounded. 
		The conditions \(\sum_{j,k} |a_{jk}|^2 < \infty\) and \(\sum_{j,k} |a_{jk}| < \infty\) are for example sufficient but not necessary.

		In special cases other ways of being given an operator can be better: 
		For an operator of rank \(n\) on a Hilbert space we would prefer to be given \(2n\) vectors, 
		since this is sufficient to determine such an operator uniquely.

		A compact self-adjoint operator is determined by the sequence of eigenvalues \(\{ \lambda_n \}\) 
		together with a orthonormal basis where it becomes a diagonal operator; 
		in practice this means that if we start with a determined ``computable'' orthonormal basis, 
		we must be given a change of basis, i.e. a unitary operator. 
		But how shall a be given a unitary operator \(U\)? 
		If \(\{ e_n \}\) are it ``computable'' orthonormal the basis can \(U\) be given by \(\{ Ue_n \}\). 
		We can require that we shall be given the sequence of (real) eigenvalues ordered by size, the largest first. 
		In this case it becomes easy to find the norm, which then equals the absolute value of its largest eigenvalue. 
		But then we get a numbering that cannot be effectively reduced to the graph, 
		i.e. we can not effectively from the graph find the largest eigenvalue. 
		In an extreme case we can say that to be given an operator is to be given the graph and the norm. 
		The point must be to reveal the effective content of classical mathematics, then the numbering must be chosen for that purpose.

\section{\(\B(\H)\) as an effective topological space}
	\input{b-of-h.tex}

\section{Dual spaces}
	\input{dual-spaces.tex}

%% file: counterexamples.tex
\subsection{\(\B(\X)\) not separable}
	Let \(\X\) be a Banach space. 
	The Banach space \(\B(\X)\) of bounded linear operators on \(\X\) is a natural starting point for a comparison between 
	higher order recursion theory and the theory of computability structures on Banach spaces.
	A difficulty for the comparison is that \(\B(\X)\) is not separable.

	\begin{Lemma}
		Let \(\H\) be a separable Hilbert space.
		Then \(\B(\H)\) separable iff \(H\) is finite dimensional.
	\end{Lemma}
	%
	%
	%

	The effective Banach spaces are thus not closed under exponentiation \(\X \mapsto \B(\X)\), so 
	\textbf{Ban} does not automatically become an ``effective cartesian closed category''. 
	But we can limit ourselves to the closure of the linear span of the effective operators. It becomes a separable closed subspace of \(\B(\X)\).

\subsection{An effective operator with non-computable norm}

	\begin{Theorem}
		There exist an effective operator on \(\ell^2\) with non-computable norm.
		There exists thus no computability structure on \(\B(\ell^2)\) where all effective operators are computable elements, since \Norm is not satisfied.
	\end{Theorem}

	\begin{proof}
		Define a multiplication operator \(a\) by
		\[
			(a \xi)_n = \left( \sum_{k=1}^{\infty} \frac{\mc{K}_n (k)}{2^k} \right) \xi_n
		\]
		where \(\mc{K}\) is a complete r.e. set and \(\mc{K}_n\) is the \(n\)th finite approximation of \(\mc{K}\).
		Then \(a\) is linear and bounded with norm \(r_\mc{K} \equiv \sum_{k=0}^{\infty} \frac{\mc{K}(k)}{2^k}\) which is a non-computable real number between 0 and 1.
		\Lin gives that \(a\) preserves computability of sequences, 
		since \(a\) multiplies a sequence of complex numbers with a computable sequence of rational numbers.
		\(a\) is continuous, and thus effective (by Sequential Effectivity).
	\end{proof}

\subsection{Ineffective operator norm on \(\B(\X)\)}

	We gives an example of in-effectivity of the norm on the space of bounded linear operators over a infinite dimensional Banach space. 
	The example is \(\X=\ell^2\), perhaps the simplest infinite dimensional Banach space.

	\begin{Prop}
		There exists no recursive \(g:\N \saa \N\) such that \((\forall e) g(e) > \|T^e\|\) for \(T^e \in \B(\ell^2)_{\op{k}}\).
	\end{Prop}

	\begin{proof}
		On the Hilbert space \(\ell^2\), an operator is given by a infinite matrix \(a_{ij}\).
		From the index \(e\) for \(T^e\) we can effectively find a index \(f(e)\) for the sequence of matrix elements, 
		and conversely, since the matrix element \(a_{ij}=(e_i|Te_j)\) where \(\{ e_n \}\) is a orthonormal basis and an e.g.s.

		For each individual \(e,x \in \N\), consider the effective operator \(\mc{O}^{ex}\) given by 
		the diagonal matrix \(a_{ss}=\varphi_{e,s}(x)\) if \(\varphi_{e,s}(x) \downarrow\) and \(a_{ss}=0\) otherwise. 
		Let \(\mc{O}^{ex}\) ha index \(h(e,x)\) where \(h\) are recursive.

		If \((\forall e) g(e) > \|T^e\|\) we have then \(g(h(e,x)) > \|T^{h(e,x)} = \mc{O}^{ex} \| \geq \varphi_e(x)\). 
		Let \(k(x)=g(h(x,x))+1\). Then \(k\) recursive, so \((\exists c) k=\varphi_c\). But then are \(\varphi_c(c) > \varphi_c(c)\), a contradiction.
	\end{proof}

	\begin{Kor}
		The norm, restricted to operators with computable norm, is a ineffective operator in \(\B(\B(\ell^2),\R)\).
	\end{Kor}

\subsection{A uncountable computability structure}

	\begin{Theorem}
		The definition of computability structures in \cite{Pour} allows for uncountably many computable elements.
	\end{Theorem}

	\begin{proof}
		Let \(\X\) be it non-separable Banach space of {\em almost periodic functions} that have orthonormal basis \(\{e^{i \lambda t} \}_{\lambda \in \R}\) 
		and norm \(\| f \|^2 = \lim_{T \saa \infty} \frac{1}{T} \int_{0}^{T} |f(x)|^2 dx\). We do not specify any computability structure for \(\X\).

		We shall show that there exists a computability structure on \(\B(\X)\) that contains all projections on one-dimensional subspaces.
		Since \(\X\) has uncountable dimension, there exists uncountably many one-dimensional projection operators on \(\X\).

		The computability structure is obtained by starting with all one-dimensional projection operators and taking
		the closure under computable linear combinations and effective limes. 
		We must show that \Norm then becomes satisfied.

		Let \((x_k)\) and \((y_k)\) be computable sequences of projection operators with rank \(1\).

		By Effective Density it suffices to consider the norm of an effective limit and the norm of a computable linear combination.

		We have \(\| \sum_{k=0}^{d(n)} \alpha_{nk} x_k + \beta_{nk} y_k \| = \op{max}_{k=0}^{d(n)} \{ \alpha_{nk}, \beta_{nk} \}\) 
		since the norm is realized in an element in the one-dimensional subspace where the weight \(\alpha\) is maximized. 
		We may assume that all \(x_k\) are distinct, 
		since the coefficient of a linear combination of a given operator equals the sum of the coefficients in the linear combination.

		If \(\lim_m^* x_{mn}\) exists then \(\|\lim_m^* x_{mn} \| =\lim_m^* \| x_{mn} \|\), 
		and the computable real numbers are closed under effective limits.
	\end{proof}

	This computability structure is too large, but also too small since it only contains compact operators.
	Effective Banach space have countable computability structures, since the computable elements are computable linear combinations of elements in the e.g.s.

	\paragraph{\(T,S \mapsto T+S\) does not preserve computable norm}
		The operator norm has many ineffective properties: also the sum of two operators with computable norm can have non-computable norm. 
		It does not help that the norm of a sum has a well known approximation known as {\em the triangle inequality}: \(| \|T\|-\|S\| | \leq \|T+S\| \leq \|T\| + \|S\|\). 
		It follows that we cannot find the distance \(d(T,S)=\|T-S\|\) between two operators with computable norm, 
		even though thus \(| \|T\|-\|S\| | \leq \|T-S\| \leq \|T\| + \|S\|\).

		\begin{Prop}
			There exist effective operators \(R = (r_{i,j})\) and \(S=(s_{i,j})\) on \(\ell^2\) with computable norms \(\|R\|\) and \(\|S\|\), 
			such that \(\|R+S\|\) is non-computable.
		\end{Prop}
		\begin{proof}
			Choose an integer \(M > \|T\|\), let \(r_{1,1}=M\), \(s_{1,1}=-M\), and for \(j>1\), \(r_{j,j} = s_{j,j} = t_{j,j}\). 
			Then \(\|R\|=\|S\|=M\), but \(R+S=T\) so \(\|R+S\|=\|T\|=r_\mc{K}\).
		\end{proof}

		\begin{Prop}
			There exist effective operators \(P\) and \(Q\) on \(\ell^2\) with
			\[ \|P\|, \|Q\|, \lim_n \|Pe_n\|, \lim_m \|Qe_m\| \in \Rk, \|P+Q\| \not\in \Rk \]
		\end{Prop}
		\begin{proof}
			Let \(P\) be given by a diagonal matrix \((p_{j,j}=e^{i\theta_j})\) where \((\theta_j)\) is a monotonically decreasing positive rational sequence with limit 
			\(\theta \in [0, \frac{\pi}{2}]\), and let Q=I (the identity operator).
			Then \(\|P\|=\|Q\|=\lim \|Qe_n\| = \lim \|Pe_n\|=1\), whereas \(\|P+Q\|=sup_n |e^{i \theta_n} + 1| = |e^{i \theta} + 1|\) is non-computable.
		\end{proof}

	The norm is defined as a supremum over \(\X\). Since the computable points in \(\X\) are dense, we can bound the supremum to vary over computable points. 
	If we do so, we see that the supremum cannot be achieved (since the operator is effective) if the norm is a non-computable real number. 
	There must exist a sequence \((x_n)\) on the unit ball such that \(\| Tx_n \| \saa \| T \|\).

%% file: b-of-h.tex
Also for general topological spaces one can study computability.

\begin{Def}[Nogina, 1966]
	A separable topological space \(S\) with a countable dense subset \(X\) and a countable basis \(U\) for subspaces the topology on \(X\), 
	with numberings \(\alpha\) of \(X\) and \(\beta\) of \(U\), 
	is called an effective topological space if
	\begin{description}
		\item[ (i)] \(\exists f \in \mc{R} \quad \forall e \quad \alpha e \in \beta f e\)
		\item[(ii)] 
			\(\exists g \in \mc{R} \quad \forall \la e_1, e_2, e_3 \ra \quad \alpha e_1 \in \beta g (e_1,e_2,e_3) 
			\delm \beta e_2 \snitt \beta e_3\)
	\end{description}
\end{Def}

In an effective topological space we shall thus effectively could find a open basic neighborhood of a computable point. Furthermore shall the basis be effectively closed under intersection.

The operator space \(\B(\H)\) are not an effective Banach space, then it is not separable in the norm topology. 
One can ask if the set \(\B(\H)\) becomes an effective topological space if we gives it a separable topology.

\begin{Def}[Strong and weak operator topology on \(\B(\H)\)]
	\(\B(\H)\) are the set of continuous linear operators on a separable Hilbert space \(\H\).

	The strong operator topology on \(\B(\H)\) are the topology induced by the semi-norms \(m_x\) (\(x \in \H\)) given by \(m_x(T) = \|Tx\|\).

	The weak operator topology on \(\B(\H)\) are the topology induced by the semi-norms
	\(m_{x,y}\) (\(x,y \in \H\)) given by \(m_{x,y}(T)=|(y|Tx)|\).
\end{Def}

\begin{Theorem}
	\(\B(\H)\) in respectively the weak and strong operator topology are effective topological spaces.
\end{Theorem}

\begin{proof}
	In these two topologies \(\B_f(\H)\), the set of operators of finite rank, is dense in \(\B(\H)\).

	An operator of finite rank can be written in the form \(z \mapsto \sum_{i=1}^r (z|y_i)x_i\) where \(r\) is the rank and \(x_i, y_i \in \H\). 
	Since \(\H\) is separable, \(\B_f(\H)\) is separable (in the norm topology, and thus also in all weaker topologies).
	Let \(\{ T^n \}\) be a countable (norm-)dense subset of \(\B_f(\H)\).

	Let \(\{ a_n \}\) be a numbering of the rational numbers \(\Q\).

	Let \(\{ \vec{x}_n \}\) be a numbering of the finite sequences from a countable dense subset of \(\H\).

	A countable basis for the strong operator topology on \(\B_f(\H)\) is given by

	\[
		B_s(T^{(n)_0}, a_{(n)_1}, \vec{x}_{(n)_2} ) = \{ T | \, \|(T-T^{(n)_0})x\| < a_{(n)_1} \quad \forall x \in \vec{x}_{(n)_2} \}
	\]

	A countable basis for the weak operator topology on \(\B_f(\H)\) is

	\[
		B_w(T^{(n)_0}, a_{(n)_1}, \vec{x}_{(n)_2} ) = \{ T | |(y|(T-T^{(n)_0})x)| < a_{(n)_1} \quad \forall x,y \in \vec{x}_{(n)_2} \}
	\]

	Condition (i): Here we have many possibilities, e.g. \(T \mapsto B_s(T,1,e_1)\) and \(T \mapsto B_w(T,1,e_1)\) 
	where \(e_1\) is a certain vector in \(\H\).

	Condition (ii): \(n,m,k \mapsto B(T^k, \op{min}(a_{(n)_1}, a_{(m)_1}), \vec{x}_{(n)_2} \union \vec{x}_{(m)_2} )\).
\end{proof}

\paragraph{Separable spaces}
	In the norm topology \(\B(\H)\) is not separable, and therefore not an effective topological space.

	In applications it is often desirable to ha a topology not only on the space they effective operators act on, 
	but also on the set of effective operators. 
	A perturbation \(T'\) of an operator \(T\) is an operator that in some sense lies near \(T\). 
	A problem can often not be solved for \(T\) directly, but for approximations to \(T\). 
	To make this precise it is desirable to make \(\B(\H)\) into a topological space. 
	If one studies the effective operators it is then desirable to make \(\B(\H)\) into an effective topological space.
	The idea that effective topological spaces must be separable can be defended based on the thought that 
	the subset of computable elements in a given set always must be countable. 
	Intuitively this is because there are only countably many computer programs in a given language with finite alphabet. 
	In practice all objects studied will often be computable. 
	A good example is the real numbers, where all numbers one encounters in analysis are computable. 
	For this to be possible the computable elements ought to form a dense set.

%% file: dual-spaces.tex
The dual space of a Banach space \(\X\) is defined by \(\X^*= \B(\X,\F)\), the space of of bounded linear operators, often called functionals, from \(\X\) to \(\F\).

The operation of taking the dual \(*\) on the class of Banach spaces does not preserve separability. 
E.g., \(\ell^1\), the space of absolutely summable complex sequences, is separable.
Its dual \((\ell^1)^* = \ell^\infty\), the space of bounded complex sequences, is non-separable: \(\{ 0,1 \}^\N \delm \ell^\infty\). 
If \(\xi,\eta \in \{ 0,1 \}^\N\), \(\xi \neq \eta\), then \(\|\xi-\eta\|_\infty = \sup_n |\xi_n - \eta_n| = 1\). 
Thus there exist uncountably many elements of \(\ell^\infty\) having mutual distance \(1\). 
A countable subset can therefore not intersect each open ball of radius 1, and thus cannot be dense.

The adjoint \(T^* \in \B(\Y^*,\X^*)\) of an operator \(T \in \B(\X,\Y)\) is defined by \((T^* \varphi)(x) = \varphi(Tx)\).
Composition of operators is an effective operation, for from \(F \alpha e = \alpha f(e)\) and \(G \alpha e = \alpha g(e)\) follows 
\(F \circ G \alpha e = F \alpha g(e) = \alpha f(g(e))\), and the set of all partial recursive functions is closed under composition.
Thus the class of effective operators is closed under adjunction, 
assuming that \(\X\), \(\Y\), \(\X^*\) and \(\Y^*\) are effective Banach spaces, which in turn requires them to be separable.

Separable spaces of operators may often be made effective. An example is \(\B_0(\X)\) when \(\X\) has the approximation property. 
Another is \(L^p(\R)\) (\(1<p<\infty\)), the space of equivalence classes of Lebesgue-measurable functions 
\(f:\R \saa \C\) such that \(\int_{-\infty}^{\infty} |f(x)|^p dx < \infty\), 
where two functions are called equivalent if they are equal on a set whose complement has measure zero. 
As shown in \cite{Pour} this can be made into an effective Banach space.
Since \(L^p\) is the dual of \(L^q\) when \(p\) and \(q\) are conjugate exponents (\(p^{-1}+q^{-1}=1\)), 
and the operator norm from \(\B(L^q,\C)\) is equal the \(L^p\)-norm \((\int_{-\infty}^{\infty} |f(x)|^p dx)^{1/p}\), 
we have that \(L^p\) is an example of an effective space of operators.

\paragraph{Analysis of the operator norm's effectivity}
	It may seem as if it was the operator norm by itself that creates trouble, 
	but from these examples of separable spaces of operators with effective norm it may seem that 
	the problem is rather tied to non-separability in general. 
	In support of this view we may point to \(\ell^\infty\), which has the supremum norm.

	In general one can say that the operator norm is expressed by a supremum over an infinite set and therefore is not effective unless 
	there is an alternate characterization, e.g. if 
	the space is separable and the operators in the countable dense subset have an especially workable norm. 
	In the case of \(\B(\X)\), the subspace spanned by the effective operators in \(\B(\X)\) is separable, 
	but that does not help make the norm effective.